\documentclass{amsart}
\usepackage[T1]{fontenc}
\newtheorem{theorem}[subsection]{Theorem}

\usepackage{amscd,amsmath,amssymb,amsthm,bm,cases,color,comment,mathrsfs,mathtools}
\usepackage{graphicx}

\theoremstyle{definition}
\newtheorem{definition}{Definition}[section]
\newtheorem{remark}[definition]{Remark}
\newtheorem{example}[definition]{Example}

\theoremstyle{plain}
\newtheorem{proposition}[definition]{Proposition}
\newtheorem{lemma}[definition]{Lemma}
\newtheorem{corollary}[definition]{Corollary}
\numberwithin{equation}{section}

\title{The bicorn curves on closed surfaces}

\date{}

\author[T.~Katayama]{Takuya Katayama}
\address{
(Takuya Katayama)
Osaka Central Advanced Mathematical Institute, 
3-3-18 Sugimoto, Sumiyoshi-ku, Osaka 558-8585, Japan
}
\email{tak.katayama0221@gmail.com}
\author[E.~Kuno]{Erika Kuno}
\address{
(Erika Kuno)
College of Engineering,
Shibaura Institute of Technology,
3-7-5 Toyosu, Koto-ku, Tokyo 135-8548, JAPAN
}
\email{e-kuno@shibaura-it.ac.jp}

\begin{document}

\begin{abstract}
This paper focuses on using the theory of bicorn curves in the context of closed surfaces to understand hyperbolic phenomena of the curve graphs of those surfaces.
We prove that the curve graph of any closed surface is $15$-hyperbolic with one exception. 
Furthermore, we provide significantly tighter bounds for the bounded geodesic image theorem, originally proven by Masur--Minsky. 
\end{abstract}

\maketitle

\section{Introduction}

Throughout this paper, by $F$ we denote a closed orientable or non-orientable surface of $\chi(F) \leq -2$. 
The {\it curve graph} $\mathscr{C}(F)$ of $F$ is a graph whose vertex set consists of the homotopy classes of essential simple closed curves on $F$. 
Two vertices of $\mathscr{C}(F)$ are connected by an edge if they can be represented disjointly on $F$. 
The graph has a usual length metric $d_{F}$ when we endow each edge with length $1$. 
An well-known inequality on the distance of the curve graph, due to Hempel \cite[Lemma 2.1]{Hem01}, is that 
$$d_{F} (\alpha, \beta) \leq 2 \log_{2} (i (\alpha, \beta)) + 2$$
 for all vertices $\alpha, \beta \in \mathscr{C}(F)$ with $i (\alpha, \beta) \geq 1$.  
Here, $i(\alpha, \beta)$ is the geometric intersection number of $\alpha$ and $\beta$. 
We improve the above inequality as follows. 

\begin{theorem}\label{log_upp_1}
Let $\alpha, \beta$ be essential simple closed curves on $F$ with $i(\alpha, \beta) \geq 2$. 
Then we have 
$$d_F(\alpha, \beta) \leq 2 \log_{3} \left( \frac{9}{4} i(\alpha, \beta) \right).  $$
\end{theorem}

Given a geodesic triangle $\Delta$ of a geodesic metric space $X$, a point of $X$ whose $\delta$-neighbourhood touches all sides of $\Delta$ is called a {\it $\delta$-center}. 
For a fixed $\delta \geq 0$, a geodesic metric space $X$ is said to be {\it $\delta$-hyperbolic} if every geodesic triangle in $X$ has a $\delta$-center. 
Masur--Minsky in \cite{MM99} proved that the curve graphs of surfaces are $\delta$-hyperbolic for some constant $\delta$ depending on the surface type. 
After Masur--Minsky's paper was published, many authors (e.g.\ \cite{A13}, \cite{BF07} \cite{B14}, \cite{CKS14}, \cite{HPW15}, \cite{K17}, \cite{MS13}, \cite{PS17}) have improved the proof and gave concrete hyperbolic constants. 
In particular, Hensel--Przytycki--Webb in \cite{HPW15} gave an extremely short proof and low bound $\delta \leq 17$ which is independent from the surface type. 
The second author generalized their result for non-orientable surfaces. 
In this paper we obtain a slightly lower hyperbolic constant for the curve graphs of closed surfaces. 

\begin{theorem}\label{hyp}
$\mathscr{C}(F)$ is 15-hyperbolic. 
\end{theorem}

To the best of the authors' knowledge the above hyperbolic constant for $\mathscr{C}(F)$ is optimal among the available estimates. 

The bounded geodesic image theorem due to Masur--Minsky \cite{MM00} is a consequence of the $\delta$-hyperbolicity of the curve graphs. 
It states that, for any surface $F$ and its subsurface $F'$, the subsurface projection of every geodesic in $\mathscr{C}(F)$ whose all members cut $F'$ has finite diameter in $\mathscr{C}(F')$ (and the upper bound only depends on the surface type of $F$). 
For orientable $F$, explicit and uniform subsurface projection bounds independent from the surface type of $F$ were given by Webb \cite{W14} and Jin \cite{Jin20}. 
However, no explicit subsurface projection bound is known for non-orientable surfaces. 
We will prove the bounded geodesic image theorem for closed surfaces including non-orientable ones with effective bounds. 

\begin{theorem}\label{bgit}
Suppose that $F'$ is a subsurface of $F$, which is not homeomorphic to a M\"obius band. 
Then for every geodesic $\mathcal{G}$ in $\mathscr{C}(F)$ whose all members cut $F'$, the following (1) and (2) hold. 
\begin{enumerate}
 \item[(1)]  $\mathrm{diam}_{F'}(\pi_{F'}(\mathcal{G})) \leq 32$ if $F'$ is non-annular. 
 \item[(2)]  $\mathrm{diam}_{F'}(\pi_{F'}(\mathcal{G})) \leq 52$ if $F'$ is annular.  
\end{enumerate}
\end{theorem}

For the definition of the subsurface projection $\pi_{F'}$, see Sections \ref{non-annular} and \ref{annular}. 

The {\it $k$-th augmented curve graph} $G_k (F)$ of $F$ is the graph such that vertices are the same as the curve graph, and that the edges are spanned by two vertices which have representatives intersecting at most $k$. 
Note that $G_{0} (F)$ is the ordinary curve graph $\mathscr{C}(F)$. 
The length metric can be introduced into $G_{k} (F)$ in the same way as for  $\mathscr{C}(F)$. 
For any $k \geq 0$, $G_{k}(F)$ is not 0-hyperbolic, because it contains a triple of pairwise adjacent vertices. 
From $\delta$-hyperbolicity of the curve graphs, it follows that $G_{k}(F)$ is $\delta'$-hyperbolic for some $\delta'$. 
As a natural consequence of the properties of bicorn curves, we prove the following. 

\begin{theorem}\label{aug_hyp}
For any $k \geq 2$, the augmented curve graph $G_k (F)$ is $8$-hyperbolic. 
\end{theorem}

For recent works on the augmented curve graphs, see e.g.\ \cite{AG22}, \cite{BN20}, \cite{G19}, \cite{P15}. 

This paper is structured as follows. 
In Section \ref{tbc}, we develop the theory of bicorn curves on closed orientable or non-orientable surfaces. 
Theorems \ref{log_upp_1}, \ref{hyp} and some results concerning bicorn curves and distance in the curve graphs will be proved in Section \ref{IND}. 
Sections \ref{non-annular} and \ref{annular} are devoted to proving Theorem \ref{bgit}. 
In the appendix we will prove Theorem \ref{aug_hyp}. 

%%%%%%The theory of bicorn curves%%%%%%%%%
\section{The bicorn curves} \label{tbc}

Inspired by the proof in \cite{HPW15}, Przytycki--Sisto \cite{PS17} used bicorn curves on closed surfaces to give a short proof of hyperbolicity of the curve graphs. 
Bicorn curves are also used in similar contexts by Rasmussen \cite{R20} and Jin \cite{Jin20}. 
Following Przytycki--Sisto, we define bicorn curves. 

\begin{definition}
Let $\alpha$ and $\beta$ be essential simple closed curves on $F$. 
Suppose that $\alpha$ and $\beta$ are in minimal position. 
An {\it $(\alpha, \beta)$-bicorn curve} $\gamma$ is a simple closed curve such that $\gamma = a \cup b$ for some 
arcs $a \subset \alpha$ and $b \subset \beta$ with $\partial a = \partial b$. 
Here, exactly one of $a$ or $b$ can be $\emptyset$, and hence both of the entire curves $\alpha$ and $\beta$ can be seen as bicorn curves between themselves. 
We call $a$ (resp.\ $b$) an {\it $\alpha$-arc} (resp. {\it $\beta$-arc}). 
Any bicorn curve is essential, because $F$ is closed. 
Also, for any pair of essential simple closed curves on $F$, only finitely many bicorn curves exist. 

An {\it arc} on a surface is an embedded closed interval in the surface. 
A {\it proper arc} on a surface with boundary is an arc whose endpoints lie in the boundary of the surface. 
If $a$ is an arc, $\partial a$ denotes the set of the endpoints of $a$, and $\mathrm{Int}(a)$ denotes $a - \partial a$. 
\end{definition}

\begin{example}
Given essential simple closed curves $\alpha$ and $\beta$ on a surface $F$ which are in minimal position, we can construct an $(\alpha, \beta)$-bicorn curve. 
Consider the cyclic graph $\Gamma_{\alpha}(\beta)$ whose geometric realization is $\beta$, and the vertices are the points of $\alpha \cap \beta$. 
Pick an edge $e$ of $\Gamma_{\alpha}(\beta)$. 
Then $e$ is a minimal arc of $\beta$ such that $e \subset \beta$ and $\partial e \subset \alpha \cap \beta$. 
Note that $e$ determines two arcs $a, a'$ of $\alpha$ with $\partial a = \partial a' = \partial e$. 
Both of $a \cup e$ and $a' \cup e$ are simple by minimality of $e$. 
Thus $a \cup e$ and $a' \cup e$ are $(\alpha, \beta)$-bicorn curves. 
Similarly, each edge of $\Gamma_{\beta}(\alpha)$ gives exactly two $(\alpha, \beta)$-bicorn curves containing the edge. 
\end{example}

We start with the following lemma.  

\begin{lemma}\label{1-3rd}
Let $\alpha$ and $\beta$ be essential simple closed curves on $F$ which are in minimal position. 
If $i(\alpha, \beta) \geq 3$, then there exists an $(\alpha, \beta)$-bicorn curve $\gamma$ such that $i(\alpha, \gamma) \leq 2$ and $i(\beta, \gamma) \leq \frac{i(\alpha, \beta)}{3}$. 
\end{lemma}
\begin{proof}
Suppose $i(\alpha, \beta) \geq 3$. 
It is enough to see that there exists three $(\alpha, \beta)$-bicorn curves $\gamma_1, \gamma_2, \gamma_3$ such that 
\begin{enumerate}
 \item[(1)] $i(\alpha, \gamma_i) \leq 2$ (for $i = 1,2,3$) and 
 \item[(2)] $\Sigma_{i=1}^{3}i(\beta, \gamma_i) \leq i(\alpha, \beta)$. 
\end{enumerate}
In order to obtain such bicorn curves, we consider the cyclic graph $\Gamma_{\alpha}(\beta)$. 
Pick two consecutive edges $e_1, e_2$ of $\Gamma_{\alpha}(\beta)$. 
Each of the arcs $b_1 = e_1$, $b_2 = e_2$ and $b_3 = e_1 \cup e_2$ determines an arc of $\alpha$, say $a_i$, whose endpoints are equal to $\partial b_i$, and $\mathrm{Int}(a_i)$ does not pass $e_1 \cup e_2$. 
Set $\gamma_{i} := a_{i} \cup b_{i}$ ($i = 1, 2, 3$). 
Then $\gamma_1, \gamma_2, \gamma_3$ are $(\alpha, \beta)$-bicorn curves. 
Since $e_1 \cup e_2$ intersects $\alpha$ thrice and since one of the intersection points can be avoided by an isotopy, it follows that $i(\alpha , \gamma_i) \leq 2$. 
Hence, $\gamma_i$ satisfies the above condition (1) (for $i=1,2,3$). 
Since $\sum_{i=1}^{3} \# (\mathrm{Int}(a_i), \beta) \leq i(\alpha, \beta) - 3$, and since a representative of $\gamma_i$ intersects $\beta$ at most once in the regular neighborhood of $e_1 \cup e_2$, and so $\gamma_1, \gamma_2, \gamma_3$ satisfy the above condition (2). 
\end{proof}

It is easy to see that any two essential simple closed curves $\alpha$ and $\beta$ intersecting twice has a bicorn curve $\gamma = a \cup b$ such that $i(\alpha, \gamma) \leq 1$ and $i(\beta, \gamma) \leq 1$. 
Here, $a$ is an $\alpha$-arc and $b$ is a $\beta$-arc. 
Moreover, the sequence of the bicorn curves, $\alpha, \gamma, \beta$, has an obvious property that $\alpha \supset a \supset \emptyset$ and $\emptyset \subset b \subset \beta$, when we regard the curves as $\alpha = \alpha \cup \emptyset$, $\gamma = a \cup b$ and $\beta = \emptyset \cup \beta$. 
The following lemma generalizes this toy case. 

\begin{lemma}\label{bicorn_path_exists_2}
Let $\alpha$ and $\beta$ be essential simple closed curves on $F$ with $i (\alpha, \beta) \geq 3$ which are in minimal position. 
Then there exists a sequence of bicorn curves $\alpha = \gamma_1, \gamma_2, \ldots, \gamma_{n-1}, \gamma_n = \beta$ satisfying the following properties (I) and (II):
\begin{enumerate}
 \item[(I)] $\alpha \supset a_2 \supset \cdots \supset a_{i-1} \supset a_i \supset \cdots \supset a_{n-1} \supset \emptyset$, and that $\emptyset \subset b_2 \subset \cdots \subset b_{i-1} \subset b_{i} \subset \cdots \subset \beta$, 
 where the $a_i$ is an $\alpha$-arc and $b_i$ is a $\beta$-arc of $\gamma_i$.  
 \item[(II)] $i(\gamma_{i-1}, \gamma_{i}) \leq 1$ for each $i \geq 2$. 
\end{enumerate}
Consequently, for all $1 \leq i \leq n$ and for all $j \leq i \leq k$, it follows that $\gamma_i$ is a $(\gamma_j, \gamma_k)$-bicorn curve. 
\end{lemma}
\begin{proof}
We construct the desired sequence inductively. 
Set $\gamma_1 := \alpha$. 
Pick an edge $e$ of $\Gamma_{\alpha} (\beta)$. 
Then $e$ determines two arcs $a, a'$ of $\alpha$ such that $a \cup a' = \alpha$, $\partial a = \partial a' = \partial e$ and $\mathrm{Int}(a) \cap \mathrm{Int} (e) = \mathrm{Int}(a') \cap \mathrm{Int} (e)= \emptyset$. 
We may assume that $a$ has the property that $\mathrm{Int}(a)$ passes thorough a vertex $v$ of $\Gamma_{\alpha} (\beta)$ which is exactly 1 away from $\partial e$ in distance. 
Set $\gamma_2:= a_2 \cup b_2$, where $a_2 = a$ and $b_2 = e$. 
Since $\gamma_1 \cap \gamma_2 = a_2$, it follows that $i(\gamma_1, \gamma_2) \leq 1$. 

Set $b_{3} := b_{2} \cup e_{v}$, where $e_{v}$ is the edge connecting $v$ and $b_{2}$. 
Since $a_2$ has the property noted above, there exist an arc $a_{3} \subset a_{2}$ such that $a_{3}$ satisfies $\partial a_{3} = \partial b_{3}$
We set $\gamma_3: = a_3 \cup b_{3}$. 
Since $\gamma_2 \cap \gamma_3$ is an arc $b_2 \cup a_3$, we have $i (\gamma_2, \gamma_3) \leq 1$. 
It is possible that $a_3$ no longer has the property similar to $a_2$. 
We will define $\gamma_4$ depending on whether $\mathrm{Int} (a_3)$ intersects $\beta - b_3$ or not. 
If $\mathrm{Int}(a_3)$ avoids $\beta - b_3$, then we have $\gamma_3 \cap \beta = b_3$ and hence $\gamma_1, \gamma_2, \gamma_3, \beta$ is the desired $(\alpha, \beta)$-bicorn sequence. 
So we may assume that $\mathrm{Int}(a_3)$ intersects $\beta - b_3$ at least once. 
Pick a vertex $v'$ of $a_3 \cap (\Gamma_{\alpha}(\beta) - b_3)$ so that $v'$ is most nearest to $b_3$ in distance. 
Let $e_1, e_2, \ldots, e_m$ be the shortest edge-path starting at $\partial b_3$ and ending at $v'$. 
Set $b_4 : = b_3 \cup e_1 \cup \cdots \cup e_m$. 
Then $b_4$ together with $a_3$ determines an arc $a_4 \subset a_3$ such that $\partial a_4 = \partial b_4$. 
Set $\gamma_4:= a_4 \cup b_4$. 
It follows that $i (\gamma_3, \gamma_4) \leq 1$, because $\gamma_3 \cap \gamma_4$ is an arc $b_3 \cup a_4$. 

Similarly, we can construct an $(\alpha, \beta)$-bicorn curve $\gamma_i$ from $\gamma_{i-1}$. 
The above process ends after a finite number of steps, because $i(\beta, \gamma_i) \leq i(\beta, \gamma_{i-1}) - 1$. 
For all $1 \leq i \leq n$ and $j \leq i \leq k$, we have that $\gamma_i$ is a $(\gamma_j, \gamma_k)$-bicorn curve, because $ a_{j} \supset a_i$ and $b_{i} \subset b_{k}$. 
\end{proof}

The sequence of $(\alpha, \beta)$-bicorn curves obtained in Lemma \ref{bicorn_path_exists_2} is called an {\it $(\alpha, \beta)$-bicorn sequence}. 
We note that a bicorn sequence is not an actual path in the curve graph. 
Furthermore, it is not a quasi-geodesic of a uniform parameter, because there exist bicorn sequenses of arbitrary length between two essential simple closed curves with distance $2$. 
For every $(\alpha, \beta)$-bicorn curve $\gamma$, we can construct an $(\alpha, \beta)$-bicorn sequence containing $\gamma$ as follows. 

\begin{lemma}\label{bc_contain}
Let $\alpha$ and $\beta$ be essential simple closed curves on $F$ with $i (\alpha, \beta) \geq 3$ which are in minimal position. 
Suppose that $\gamma$ is an $(\alpha, \beta)$-bicorn curve. 
Then there exists an $(\alpha, \beta)$-bicorn sequence $\mathcal{B}$ such that $\gamma \in \mathcal{B}$. 
\end{lemma}
\begin{proof}
We use the argument in the proof of Lemma \ref{bicorn_path_exists_2}. 
By extending the $\beta$-arc and shortening the $\alpha$-arc of $\gamma$ repeatedly, we have a sequence of $(\alpha, \beta)$-bicorn curves $\gamma, \gamma_2, \ldots, \gamma_{n-1}, \gamma_{n} =\beta$ with the properties (I) and (II) in Lemma \ref{bicorn_path_exists_2}. 
Also, by extending the $\alpha$-arc and shortening the $\beta$-arc of $\gamma$ repeatedly, we have a sequence of bicorn curves $\gamma, \gamma_{-2}, \ldots, \gamma_{1-m}, \gamma_{-m} = \alpha$ with the properties (I) and (II). 
Then $\alpha = \gamma_{-m}, \ldots \gamma_{-2}, \gamma, \gamma_{2}, \ldots, \gamma_{n} =\beta$ is the desired sequence. 
\end{proof}

Three essential simple closed curves $\alpha, \beta, \delta$ is called {\it sensible} if $\alpha \cap \beta \cap \delta = \emptyset$. 
As Przytycki--Sisto proved in \cite[Lemma 2.6]{PS17}, any triple of bicorn curves  is slim in the following sense. 

\begin{lemma}\label{slimbicorn}
Let $\alpha, \beta, \delta$ be sensible essential simple closed curves on $F$ mutually in minimal position. 
Suppose that $\gamma = a \cup b$ is an $(\alpha, \beta)$-bicorn curve with $i(\gamma, \delta) \geq 3$, where $a$ (resp.\ $b$) is an $\alpha$-arc (resp.\ $\beta$-arc) of $\alpha$ (resp.\ $\beta$). 
Then there exists an essential simple closed curve $\gamma'$ such that $i(\gamma, \gamma') \leq 2$, and that $\gamma'$ is either an $(\alpha, \delta)$-bicorn curve whose $\alpha$-arc is contained in $\mathrm{Int}(a)$, or $(\beta, \delta)$-bicorn curve whose $\beta$-arc is contained in $\mathrm{Int}(b)$. 
\end{lemma}
\begin{proof}
The curve $\delta$ can be seen as a cyclic graph $\Gamma_{\gamma}(\delta)$ of length $\geq 3$ with the vertex set equal to $(a \cup b) \cap \delta$. 
By changing the role of $a$ and $b$ if necessary, we can find a minimal arc $\delta'$ of $\Gamma_{\gamma} (\delta)$ such that $i(\mathrm{Int}(\delta'), a) = 0$, $i(\mathrm{Int}(\delta'), b) \leq 1$ and $\partial \delta' \subset a$. 
Note that $\delta'$ consists of either two adjacent edges or a single edge. 
The closed curve $\alpha$ decomposes into the union of two arcs $a', a''$ whose endpoints are equal to those of $\delta'$. 
We may assume $a' \subset a$. 
It follows from the sensibility of $\alpha, \beta, \delta$ that $a' \subset \mathrm{Int} (a)$.  
This implies $a \cap a' = a'$ and $b \cap a' = \emptyset$. 
Then $\gamma' = \delta' \cup a'$ is the desired $(\alpha, \delta)$-bicorn curve.  
In fact, for some curve $\overline{\gamma'}$ isotopic to $\gamma'$, we have $\# (a \cap \overline{\gamma'}) \leq 1$ and $\# (b \cap \overline{\gamma'}) \leq 1$. 
\end{proof}

We transform the above lemma into a more convenient result.  

\begin{proposition}\label{slimbicorn_2}
Let $\alpha, \beta, \delta$ be sensible essential simple closed curves on $F$ mutually in minimal position. 
Then there exists an $(\alpha, \delta)$-bicorn curve $\eta$ and $(\beta, \delta)$-bicorn curve $\varepsilon$ such that $i(\eta, \varepsilon) \leq 2$, and that either $\eta$ or $\varepsilon$ intersects an $(\alpha, \beta)$-bicorn curve at most twice. 
\end{proposition}
\begin{proof}
Let $\mathcal{B}$ be an $(\alpha, \beta)$-bicorn sequence obtained in Lemma \ref{bc_contain}. 
Since $\alpha$ is an $(\alpha, \delta)$-bicorn curve and since $\beta$ is a $(\beta, \delta)$-bicorn curve, Lemma \ref{slimbicorn} implies that there exists a pair of consecutive $(\alpha, \beta)$-bicorn curves $\gamma_{K} = a_{K} \cup b_{K}, \gamma_{K+1} = a_{K+1} \cup b_{K+1} \in \mathcal{B}$ such that $\gamma_{K}$ (resp.\ $\gamma_{K+1}$) intersects an $(\alpha, \delta)$-bicorn curve $\gamma_K' = a_{K}' \cup d_{K}'$ (resp.\ $(\beta, \delta)$-bicorn curve $\gamma_{K+1}' = b_{K+1}' \cup d_{K+1}'$) at most twice. 
Here, $a_{i}$ and $a_{i}'$ are $\alpha$-arcs, $b_{i}$ and $b_{i}'$ are $\beta$-arcs, and $d_{i}'$ is a $\delta$-arc (for $i=K, K+1$). 

We first treat the case where the $\delta$-arc $d_{K}'$ of $\gamma_{K}'$ intersects $b_{K+1}'$ at least twice. 
We may assume that $\mathrm{Int} (d_{K}')$ intersects $\mathrm{Int}(b_{K+1}')$ at least twice by using an isotopy. 
Then we can obtain a $(\beta, \delta)$-bicorn curve $\varepsilon$ from $d_{K}'$ and $b_{K+1}'$ by choosing a minimal arc $b \subset b_{K+1}'$ so that the interior of $b$ is disjoint from $d_{K}'$, and that $\partial b \subset d_{K}'$. 
We note that $\varepsilon$ can be $\gamma_{K+1}'$. 
Since $\gamma_{K}' \cap \varepsilon = \partial b$, we have $i(\gamma_{K}', \varepsilon) \leq 1$. 
Set $\eta := \gamma_{K}'$. 
Since $i(\eta, \varepsilon) \leq 1$ and $i (\eta, \gamma_K) \leq 2$, the assertion holds. 

The case where the $\delta$-arc $d_{K+1}'$ of $\gamma_{K+1}'$ intersects $a_{K}'$ at least twice can be treated similarly. 
Construct an $(\alpha, \delta)$-bicorn curve $\eta$ from $d_{K+1}'$ and $a_{K}'$ by choosing a minimal arc $a$ of $a_{K}'$ so that the interior of $a$ is disjoint from $d_{K+1}'$, and that $\partial a \subset d_{K+1}'$. 
By setting $\varepsilon = \gamma_{K+1}'$, we have $i(\eta, \varepsilon) \leq 1$ and $i (\varepsilon, \gamma_{K+1}) \leq 2$. 
Hence, the assertion holds. 

Suppose that $d_{K}'$ (resp.\ $d_{K+1}'$) intersects $b_{K+1}'$ (resp.\ $a_{K}'$) at most once. 
We shall prove $i(\gamma_{K}', \gamma_{K+1}') \leq 2$. 
If it is true, then by setting $\eta = \gamma_{K}'$ and $\varepsilon = \gamma_{K+1}'$ the assertion holds. 
From the construction of bicorn sequences in Lemma \ref{bicorn_path_exists_2}, it follows that $a_{K} \cap b_{K+1} = \partial a_{K} \cup \partial b_{K+1}$. 
Since $a_{K}' \subset \mathrm{Int} (a_K)$ and $b_{K+1}' \subset \mathrm{Int} (b_{K+1}')$ (see the assertion in Lemma \ref{slimbicorn}), we have that $a_{K}' \cap b_{K+1}' = \emptyset$. 
By our assumption, $\# (a_{K}', d_{K+1}') \leq 1$ and $\# (d_{K}' \cap b_{K+1}') \leq 1$.  
Since $d_{K}'$ intersects $b_{K+1}'$ at most once, $d_{K}' \not \supset d_{K+1}'$. 
Similarly, $d_{K+1}' \not \supset d_{K}'$. 
If $d_{K}' \cap d_{K+1}'$ is empty, we have $i(\gamma_{K}', \gamma_{K+1}') \leq 2$. 
So we consider the case where $d_{K}' \cap d_{K+1}'$ is not empty. 
Then $d_{K}' \cap d_{K+1}'$ must be an arc which is neither $d_{K}'$ nor $d_{K+1}'$. 
Hence, one endpoint of $a_{K}'$ (resp.\ $b_{K+1}'$) lies in $d_{K+1}'$ (resp.\ $d_{K}'$). 
This contributes $\# (a_{K}', d_{K+1}')$ (resp.\ $\# (d_{K}' \cap b_{K+1}')$), and hence we have $i(\gamma_{K}', \gamma_{K+1}') \leq 2$. 
As noted above, $\eta = \gamma_{K}'$ and $\varepsilon = \gamma_{K+1}'$ are the desired bicorn curves. 
\end{proof}

%%%%%int number and dist%%%%%%
\section{Intersection number and distance between curves} \label{IND}

The logarithmic upper bound due to Hempel is somewhat reasonable, because it follows from a Masur--Minsky's result \cite[Proposition 4.6]{MM99} that for any pseudo-Anosov element $\varphi$ and any essential simple closed curve $\alpha$ on an orientable surface $S$, $d_S (\alpha, \varphi^{n}(\alpha))$ grows linearly while $i(\alpha, \varphi^n(\alpha))$ grows exponentially as $n \rightarrow \infty$. 
However, the upper bound attains the length of geodesics in $\mathscr{C}(S)$ if and only if $i(\alpha, \beta) = 1$ (this can be easily verified by Theorem \ref{log_upp_1}). 
In order to estimate the distances in  curve graphs, we desire to yield a more sharp inequality. 
To this end we first see the following inequality. 

\begin{lemma}\label{start_log}
Suppose that $\alpha, \beta$ are essential simple closed curves on $F$ with $2 \leq i(\alpha, \beta) \leq 3$. 
Then $d_{F} (\alpha, \beta) \leq i(\alpha, \beta)$. 
\end{lemma}
\begin{proof}
{\bf Case $i(\alpha, \beta) = 2$.} 
In this case $\chi(\alpha \cup \beta) = -2$. 
Since $\chi(F) \leq -2$, the exterior of $\alpha \cup \beta$ has an essential simple closed curve. 
This implies $d_{F}(\alpha, \beta) \leq 1+1 = 2$. 

{\bf Case $i(\alpha, \beta) = 3$.} 
Suppose that the regular neighborhood of $\alpha \cup \beta$ is orientable. 
The regular neighborhood $U$ of  $\alpha \cup \beta$ is not a sphere, because the intersection number in a sphere is even. 
Since $\chi(\alpha \cup \beta) = -3$, the regular neighborhood $U$ is homeomorphic to either a torus with three boundary components or an orientable surface of genus $2$ with one boundary component. 
In both cases we can find an essential closed curve in $F \setminus U$ as follows. 
If $U$ is homeomorphic to an orientable surface of genus $2$ with one boundary component, then $F \setminus U$ is not a disk, because a theroem due to Aougab--Huang \cite[Theorem 2.17]{AH15} states that no pair of essential simple closed curves with the intersection number $\leq 3$ fills a closed orientable surface of genus $2$. 
Hence, the total surface $F$ is either an orientable surface of genus $\geq 3$ or a non-orientable surface by our assumption. 
This implies that an essential simple closed curve can be found in the exterior of $U$. 
The case where $U$ is homeomorphic to a three-holed torus can be treated similarly. 
Therefore we have $d_{F}(\alpha, \beta) \leq 2$. 

We next suppose that the regular neighborhood of $\alpha \cup \beta$ is non-orientable. 
It is enough to find an $(\alpha, \beta)$-bicorn curve $\gamma$ on $F$ such that $i(\alpha, \gamma) \leq 2$ and $i(\beta, \gamma) = 0$. 
Then we have $d_{F}(\alpha, \gamma) \leq 2$ and  $d_{F}(\beta, \gamma) \leq 1$, and so $d_{F}(\alpha, \beta) \leq 3$. 
To this end, we write the components of $\beta - \alpha$ by $b_i$ ($i=1,2,3$). 
Note that the regular neighborhood $N_i$ of $\beta - b_i$ is a topological disk in $F$. 
Since the regular neighborhood of $\alpha \cup \beta$ is non-orientable, there exists a component $c$ of $(\alpha \cup \beta) - (\alpha \cap \beta)$ such that the regular neighborhood of $(\beta - b_3) \cup c$ is homeomorphic to a M\"obius band. 
By changing the role of $\alpha$ and $\beta$, if necessary, we may assume that $c = b_3$. 
Since $i(\alpha, \beta) = 3$, we have a component $a$ of $\alpha - \beta$ such that $a$ connects the distinct components of $N_3 - \beta$. 
It follows that at least one of $a \cup b_3$, $a \cup b_3 \cup b_1$ and $a \cup b_3 \cup b_2$ is the desired $(\alpha, \beta)$-bicorn curve $\gamma$ (See Figure \ref{g843}).  
\begin{figure}
\centering
\includegraphics[width=5cm]{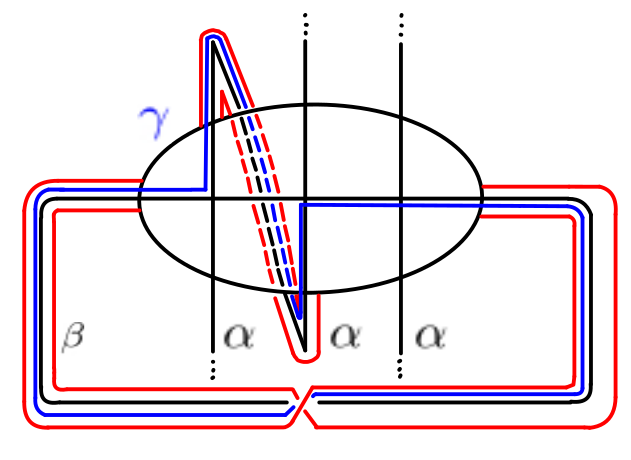}
\caption{This figure illustrates a closed curve which intersects $\alpha$ at most twice and is disjoint from $\beta$.  
The closed curve $\alpha$ is drawn as vertical black lines. 
A black square in the figure is the closed curve $\beta$. 
An oval in the middle is the regular neighborhood $N_3$ of $\beta - b_3$ in $F$. 
The arc $a$ connects not only the distinct components of $N_3 - \beta$, but also two vertical lines of $\alpha$. 
A portion of the collar of $a$ is drawn as the red band passing behind the oval. 
Similarly, the twisted red band depicts a portion of the collar of $b_3$.  
The blue closed curve is a representative of $\gamma = a \cup b_3 \cup b_2$. 
  \label{g843}}
\end{figure}
So we are done. 
\end{proof}

Let us prove Theorem \ref{log_upp_1}. 

\begin{proof}[Proof of Theorem \ref{log_upp_1}] 
Induction on $i(\alpha, \beta)$. 

{\bf Case  $2 \leq i(\alpha, \beta) \leq 3$.} Since $i (\alpha, \beta) < 2 \log_{3} ( \frac{9}{4} i(\alpha, \beta))$, Lemma \ref{start_log} directly implies the assertion. 

{\bf Case $4 \leq i(\alpha, \beta) \leq 8$.} 
By Lemma \ref{1-3rd}, there exists an essential simple closed curve $\gamma$ such that $i(\alpha, \gamma) \leq 2 < \frac{8}{3}$ and $i(\alpha, \gamma) \leq 1$. 
Then by the triangle inequality, we have $d_{F} (\alpha, \beta) \leq d_{F}(\alpha, \gamma) + d_{F}(\gamma, \beta) \leq 2 + 2 \leq 4 \leq 2 \log_3 ( \frac{9}{4} i(\alpha, \beta) )$. 
In particular, when $i (\alpha, \beta) = 4$, our upper bound is equal to $ 2 \log_3 ( \frac{9}{4}i(\alpha, \beta) )  = 2 \log_{3} 9 = 4$. 

{\bf Case $i(\alpha, \beta) \geq 9$.} 
By Lemma \ref{1-3rd}, there exists an essential simple closed curve $\gamma$ such that $i(\alpha, \gamma) \leq 2$ and $i(\beta, \gamma) \leq \frac{i(\alpha, \beta)}{3}$. 
By induction hypothesis we have the following: 
\begin{eqnarray*}
 d_{F}(\alpha, \beta) & \leq &d_{F} (\alpha, \gamma) + d_{F}(\gamma, \beta) \\
                       &\leq & 2 + d_{F}(\beta, \gamma) \\
                       &\leq & 2 + 2 \mathrm{log}_{3} \left( \frac{9}{4} i(\beta, \gamma) \right) \\
                       &\leq & 2 +  2 \mathrm{log}_{3} \left(\frac{9}{4} \frac{i (\alpha, \beta)}{3} \right). \\
\end{eqnarray*}
The last formula is equal to the desired bound $2 \mathrm{log}_{3} \left(\frac{9}{4} i (\alpha, \beta)\right)$. 
\end{proof}

From the proof of Theorem \ref{log_upp_1}, we have the following. 

\begin{proposition}
Suppose that $\alpha, \beta$ are essential simple closed curves on $F$ with $i(\alpha, \beta) \leq 8$. 
Then $d_{F} (\alpha, \beta) \leq 4$. 
\label{8to4}
\end{proposition}

We can estimate the distance between bicorn sequences and paths by using slimness of bicorn curves as follows. 

\begin{lemma}\label{dist_bc}
Let $x_0, x_1, \ldots, x_m$ be a sequence of essential simple closed curves on $F$ with $i(x_i, x_{i+1}) = 0$ $(1 \leq i \leq m-1)$ and $2^{n-1} < m \leq 2^{n}$ for some positive integer $n$. 
Suppose that $b$ is an $(x_0, x_m)$-bicorn curve. 
Then there exists a number $l$ such that $d_F(b, x_l) \leq 2n - 1 = 2 \lceil \mathrm{log}_2 (m) \rceil - 1$,  where $\lceil r \rceil$ denotes the ceil of a real number $r$. 
\end{lemma}
\begin{proof}
Induction on $n$. 
We assume $m = 2^n$, because the other cases can be treated similarly. 
Suppose $n=1$. 
Since $b$ is an $(x_0, x_m)$-bicorn curve, $x_1$ must be disjoint from $b$. 
Namely, we have $d_{F}(b, x_1) = 1$. 

For any $n$, by applying Lemma \ref{slimbicorn} to $\alpha = x_0$, $\beta = x_{2^n}$, $\delta = x_{2^{n-1}}$, we have a bicorn curve $b'$ between $x_0$ and $x_{2^{n-1}}$ (or between $x_{2^n}$ and $x_{2^{n-1}}$) such that $d_{F}(b, b') \leq 2$. 
The induction hypothesis together with the triangle inequality enables us to obtain the desired inequality. 
\end{proof}

The following proposition plays key roles in the proofs of Theorems \ref{hyp} and \ref{bgit}. 

\begin{proposition}\label{12ball_prop}
Suppose that $\mathcal{G}$ is a geodesic connecting essential simple closed curves $\alpha$ and $\beta$ in $\mathscr{C}(F)$, and $\mathcal{B}$ is an $(\alpha, \beta)$-bicorn sequence. 
Then $\mathcal{B}$ is contained in the $13$-neighborhood of $\mathcal{G}$. 
\end{proposition}
\begin{proof}
Our proof is almost the same as in the proofs of \cite[Proposition 4.2]{HPW15} and \cite[Proposition 3.4]{CJM21}. 
Let $b^{*}$ be a bicorn curve in $\mathcal{B}$ which is the furthest vertex from $\mathcal{G}$. 
Set $k:= d_{F}(b^{*}, \mathcal{G})$. 
We shall show the inequality $k \leq 2 \lceil \mathrm{log}_2 (8k + 2) \rceil - 1$, from which we can deduce $k \leq 13$. 

We first consider the case where both of $\alpha$ and $\beta$ are contained in the closed $(2k - 1)$-neighborhood of $b^{*}$. 
Then the length of $\mathcal{G}$ is not more than $4k - 2$. 
By Lemma \ref{dist_bc}, we have $k \leq \lceil \mathrm{log}_2 (4k - 2) \rceil$, and thus the assertion holds. 

We next treat the case where $\alpha$ is contained in the $(2k- 1)$-neighborhood of $b^{*}$, and $\beta$ is not. 
Since the distance between consecutive bicorn curves is at most $2$, there exists a bicorn curve $\gamma_j$ in $\mathcal{B} \cap (S(b^{*}; 2k) \cup S(b^{*}; 2k + 1))$ succeeding to $b^{*}$. 
Here, $S(b^*; r)$ is the set of the vertices of $\mathscr{C}(F)$ with distance $r$ from $b^{*}$. 
Pick a shortest geodesic $\mathcal{G}_j$ connecting $\gamma_j$ and $\mathcal{G}$, and let $L$ denote the concatenation of $\mathcal{G}_j$ and the subpath of $\mathcal{G}$ which issues from $\alpha$ and ends at $\mathcal{G}_j$. 
Note that the length of $L$ is not more than $k + (2k-1+2k+1  + k) = 6k$. 
Since $b^{*}$ is a $(\alpha, \gamma_j)$-bicorn curve, we can apply Lemma \ref{dist_bc} to $b^{*}$ and $L$, and so we have $d_{F}(\gamma^{*}, L) \leq \lceil  2\mathrm{log}_2 (6k) \rceil$. 
By $l^{*}$ we denote a curve in $L$ which is the most nearest to $b^{*}$. 
If $l^{*} \in \mathcal{G}$, then it follows that $k \leq d_F(b^{*}, l^{*}) \leq \lceil \mathrm{log}_2 (6k) \rceil$, and therefore the assertion holds. 
Otherwise $l^{*} \in \mathcal{G}_j$. 
Since $2k \leq d_{F}(b^{*}, \gamma_j) \leq d_{F}(b^{*}, l^{*}) + d_{F}(l^{*}, \gamma_j) \leq d_{F}(b^{*}, l^{*}) + k$, we have that $d_{F}((b^{*}, l^{*}) \geq k$. 
Therefore we have $k \leq \lceil \mathrm{log}_2 (6k) \rceil $ again, and the assertion holds. 
In the case where $\beta$ is contained in the closed $(4k-1)$-ball of $b^{*}$, and $\alpha$ is not, we can argue similarly as above. 

Finally, we assume that neither $\alpha$ nor $\beta$ is contained in the closed $(2k-1)$-ball of $b^{*}$.  
Pick a bicorn curve $\gamma_i$ in $\mathcal{B}$ preceding to $b^{*}$ and a bicorn curve $\gamma_j$ in $\mathcal{B}$ suceeding to $b^{*}$, which are contained in $S(b^{*}; 2k) \cup S(b^{*}; 2k + 1)$. 
In this case, $b^{*}$ lies between $\gamma_i$ and $\gamma_j$. 
Let $\mathcal{G}_i$ (resp.\ $\mathcal{G}_j$) be a geodesic of length $\leq k$ connecting $\gamma_i$ (resp.\ $\gamma_j$) and $\mathcal{G}$. 
By $g_i$ (resp.\ $g_j$) we denote the endpoint of $\mathcal{G}_i$ (resp.\ $\mathcal{G}_j$) which lies in $\mathcal{G}$. 
By $L$ we denote the concatenation of $\mathcal{G}_i$, $\mathcal{G}_j$ and the subpath $\mathcal{G}'$ of $\mathcal{G}$ which ends at $g_i$ and $g_j$. 
We note that the endpoints of $L$ are $\gamma_i$ and $\gamma_j$, and the length of $L$ is at most $k + (k+(2k+1)+(2k+1)+k) + k = 8k + 2$. 
Applying Lemma \ref{dist_bc} to $b^{*}$ and $L$, it follows that $d_{F}(b^{*}, L) \leq  2\lceil \mathrm{log}_2 (8k+2) \rceil - 1$. 
Pick a curve $l^{*}$ in $L$ which is the most nearest to $b^{*}$. 
In the case where $l^{*} \in \mathcal{G}'$, it holds that $k = d_{F}(b^{*}, l^{*}) \leq 2\lceil \mathrm{log}_2 (8k+2) \rceil - 1$, and the assertion holds. 
The case where $l^{*} \in \mathcal{G}_i \cup \mathcal{G}_j$ can be treated as follows. 
Assume that $l^{*} \in \mathcal{G}_i$ (the other case can be treated similarly). 
Since $2k \leq d_{F}(b^{*}, \gamma_i) \leq d_{F}(b^{*}, l^{*}) + d_{F}(l^{*}, \gamma_i) \leq d_{F}(b^{*}, l^{*}) + k$, it holds that $d_{F}(b^{*}, l^{*}) \geq k$. 
On the other hand, since  $l^*$ is contained in $L$, it follows that $d_{F}(b^{*}, l^{*}) \leq 2 \lceil \mathrm{log}_2 (8k+2) \rceil - 1$. 
Therefore, we have $k \leq 2\lceil \mathrm{log}_2 (8k+2) \rceil - 1$, as desired. 
\end{proof}

By combining Lemma \ref{bc_contain} and Proposition \ref{12ball_prop}, we can see that all bicorn curves are contained in some small neighborhood of a geodesic. 

\begin{corollary}\label{12ball_cor}
Suppose that $\mathcal{G}$ is a geodesic joining $\alpha$ and $\beta$ in $\mathscr{C}(F)$, and $\gamma$ is an $(\alpha, \beta)$-bicorn curve.  
Then $\gamma$ is contained in the $13$-neighbourhood of $\mathcal{G}$. 
\end{corollary}

Here, we give a proof of Theorem \ref{hyp}. 

\begin{proof}[Proof of Theorem \ref{hyp}]
Let $\Delta (\alpha, \beta, \delta)$ be a geodesic triangle in $\mathscr{C}(F)$ with apices $\alpha, \beta, \gamma$. 
By Proposition \ref{slimbicorn_2}, we have an $(\alpha, \delta)$-bicorn curve $\eta$ and $(\beta, \delta)$-bicorn curve $\varepsilon$ such that $i(\eta, \varepsilon) \leq 2$, and that either $\eta$ or $\varepsilon$ intersects an $(\alpha, \beta)$-bicorn curve at most twice. 
We may assume that $\eta$ intersects an $(\alpha, \beta)$-bicorn curve, say $\theta$, at most twice. 
It follows that $d_{F}(\eta, \varepsilon) \leq 2$ and $d_{F}(\eta, \theta) \leq 2$. 
By Corollary \ref{12ball_cor}, we have $d_{F}(\eta, [\alpha, \delta]) \leq 13$. 
Here, $[\alpha, \delta]$ is the side of $\Delta (\alpha, \beta, \delta)$ connecting $\alpha$ and $\delta$. 
Similarly, $d_{F}(\varepsilon, [\beta, \delta]) \leq 13$ and $d_{F} (\theta, [\alpha, \beta]) \leq 13$. 
Hence, $\eta$ is a $15$-center of $\Delta (\alpha, \beta, \delta)$. 
\end{proof}

\begin{corollary}
Under the setting of Proposition \ref{12ball_prop}, the geodesic $\mathcal{G}$ is contained in the $26$-neighborhood of the bicorn sequence $\mathcal{B}$. 
\label{Hausdorff_dist_cor}
\end{corollary}
\begin{proof}
Pick any vertex $g$ in $\mathcal{G}$. 
We may assume that $g$ is not contained in the $13$-neighborhood of $\mathcal{B}$. 
Pick the vertices $g_i$ and $g_j$ of $\mathcal{G}$ so that $g_i$ and $g_j$ are exactly $13$ away from $\mathcal{B}$ and nearest to $g$. 
Since all bicorn curves are contained in the 13-neighbourhood of $\mathcal{G}$, there exist bicorn curves $\gamma_i, \gamma_j \in \mathcal{B}$ such that $d_{F}(\gamma_i, g_i) \leq 13$ and $d_{F}(\gamma_j, g_j) \leq 13$, and that $d_{F}(\gamma_i, \gamma_j) \leq 2$. 
By the triangle inequality, we have $d_{F} (g_i, g_j) \leq 13+ 13+ 2 = 28$. 
This implies $d_{F} (g, \gamma_k) \leq \frac{28}{2} + 12 = 26$ for $k = i$ or $k = j$, as desired. 
\end{proof}

\begin{remark} \label{r_haus}
An alternative proof of Proposition \ref{12ball_prop} can be obtained from the Bowditch's criterion \cite[Proposition 3.1]{B14}. 
We can estimate the Hausdorff distance between a bicorn sequence and a geodesic as follows. 
Take $G_2(F)$ to be the second augmented curve graph of $F$. 
By \cite[Proposition 3.1]{B14}, the Hausdorff distance in $G_2 (F)$ is bounded by $17$. 
Since the identity map from $G_2(F)$ to $\mathscr{C}(F)$ is 2-Lipschitz (by Theorem \ref{log_upp_1}), the Hausdorff distance in the curve graph is bounded by $34$. 
We note that this is larger than the bound given in Corollary \ref{Hausdorff_dist_cor}. 
\end{remark}

%%%%The bounded Geodesic Image Theorem%%%%
\section{An alternative proof of the bounded geodesic image theorem: non-annular projections} \label{non-annular}

We say that $F'$ is a {\it subsurface} of $F$ if $F'$ is a compact two-dimensional submanifold of $F$. 
Suppose that $F'$ is a subsurface of $F$ which is neither homeomorphic to a M\"obius band nor an annulus. 
The {\it arc and curve graph}  $\mathscr{A}\mathscr{C} (F')$ of $F'$ is defined to be the graph whose vertices are the homotopy classes of proper arcs and closed curves on $F'$ which are essential in $F'$, and two vertices span an edge if they have disjoint representatives. 
Here, any homotopy of two proper arcs preserve the boundary components of $F'$. 
However, the endpoints of proper arcs can slide the boundary of $F'$ via homotopy. 

The {\it subsurface projection} $\pi_{F'} \colon \mathscr{C}(F) \to \mathfrak{P}(\mathscr{A}\mathscr{C} (F'))$ can be defined as follows. 
Let $\alpha$ be an essential simple closed curve on $F$ which is transversal to $\partial F'$. 
By deleting bigons, $\alpha$ can be homotoped to an essential simple closed curve $\alpha'$ such that $\alpha'$ and the components of $\partial F'$ are in minimal position. 
We set  $\pi_{F'} (\alpha) = F' \cap \alpha'$. 
If $\alpha'$ is disjoint from $F'$, the subsurface projection of $\alpha$ must be empty. 
Otherwise, it is either $ \alpha' $ itself or a set of proper arcs on $F'$. 
We say that $\alpha$ {\it cut} $F'$ if $\pi_{F'} (\alpha)$ is not empty. 
We denote the usual length metric of $\mathscr{A}\mathscr{C}(F')$ by $d_{F'}$. 

We first show a criterion for ensuring that the subsurface projections of an essential simple closed curve and a bicorn curve derived from it share a common proper arc. 

\begin{lemma}\label{bc_proj}
Let $\gamma_1, \gamma_2$ be essential simple closed curves on $F$, and let $\beta$ be a component of $\partial F'$. 
Suppose that $b = a_1 \cup b_1$ is a $(\gamma_1, \gamma_2)$-bicorn curve, where $a_1$ and $b_1$ are arcs of $\gamma_1$ and $\gamma_2$, respectively. 
If $a_1$ (resp.\ $b_1$) intersects $\beta$ at least thrice, then $\pi_{F'}(\gamma_1) \cap \pi_{F'}(b)$ (resp.\ $\pi_{F'}(\gamma_2) \cap \pi_{F'}(b) )$ is not empty.  
\end{lemma}
\begin{proof}
Assume that $a_1$ intersects $\beta$ at least thrice. 
Then we have three arcs of $a_1$ which issue from $\beta$ and are contained in $F'$. 
Since $\# \partial b_1= 2$, at most two such arcs end at the outside of $\beta$. 
Hence, there exists an arc $a' \subset a_1$ such that $(a', \partial a') \subset (F', \partial F')$. 
This clearly implies that $a' \in \pi_{F'}(\gamma_1) \cap \pi_{F'}(b)$. 
\end{proof}

\begin{lemma} \label{bgit_far_away}
Let $\beta$ be a component of $\partial F'$. 
Suppose that $\mathcal{G}$ is a geodesic in $\mathscr{C}(F)$ whose all vertices $\alpha$ satisfy $d_F(\alpha, \beta) \geq 18$. 
For every pair of vertices $\gamma_1, \gamma_2 \in \mathcal{G}$ and every component $c$ of $\pi_{F'} (\gamma_1)$, there exists a component $c'$ of $\pi_{F'} (\gamma_2)$ such that $d_{F'} (c, c') \leq 3$. 
\end{lemma}
\begin{proof}
Let $\gamma_1, \gamma_2$ be any vertex of $\mathcal{G}$, and $c$ a component of $\pi_{F'} (\gamma_1)$. 
We have a $(\gamma_1, \gamma_2)$-bicorn sequence $\mathcal{B}$ guaranteed by Lemma \ref{bicorn_path_exists_2}. 
Since $\mathcal{B}$ is contained in the $13$-neighborhood of $\mathcal{G}$, each $(\gamma_1, \gamma_2)$-bicorn curve $b$ is at least $5$ away from $\beta$. 
By Proposition \ref{8to4}, we have that $i (\beta, b) \geq 9$.  
Therefore, either the $\gamma_1$-arc or $\gamma_2$-arc of each $(\gamma_1, \gamma_2)$-bicorn curve intersects $\beta$ at least thrice. 
Also, both of $\gamma_1$ and $\gamma_2$ intersect $\beta$ at least thrice. 
Hence, we can find a pair of consecutive $(\gamma_1, \gamma_2)$-bicorn curves $x_K$ and $x_{K+1}$ in $\mathcal{B}$ such that both of the $\gamma_1$-arc of $x_K$ and $\gamma_2$-arc of $x_{K+1}$ intersects $\beta$ at least thrice. 
By Lemma \ref{bc_proj}, $\pi_{F'}(\gamma_1) \cap \pi_{F'}(x_K)$ (resp.\ $\pi_{F'}(\gamma_2) \cap \pi_{F}(x_{K+1})$) contains an essential proper arc $c_{K}$ (resp.\ $c_{K+1}$). 
Since $i(x_{K}, x_{K+1}) \leq 2$, we may assume that $\# (c_{K} \cap c_{K+1}) \leq 2$. 
Then by concatenating ends of $c_{K}$ and $c_{K+1}$, we can construct a unicorn arc $u$ between $c_{K}$ and $c_{K+1}$ (see \cite{HPW15} for the detail of unicorn arcs). 
Since $u$ has a representative disjoint from $c_{K} \cup c_{K+1}$, we have $d_{F'}(c_{K}, c_{K+1}) \leq 2$. 
Since any pair of the elements in $\pi (\gamma_i)$ is adjacent, we have $d_{F'}(c, c_{K+1}) \leq d_{F'} (c, c_{K}) + d_{F'} (c_{K}, c_{K+1}) \leq 1 + 2$, as desired. 
\end{proof}

\begin{lemma} \label{bgit_far_away_2}
Let $\beta$ be a component of $\partial F'$, $t$ an integer between $4$ and $8$. 
Suppose that $\mathcal{G}$ is a geodesic of length $t$ in $\mathscr{C}(F)$ and issues from a vertex $g \in \mathscr{C}(F)$ with $d_{F}(g, \beta) \geq 10 + t$. 
For every pair of vertices $\gamma_1, \gamma_2 \in \mathcal{G}$ and every component $c$ of $\pi_{F'} (\gamma_1)$, there exists a component $c'$ of $\pi_{F'} (\gamma_2)$ such that $d_{F'} (c, c') \leq 3$. 
\end{lemma}
\begin{proof}
Pick any vertices $h_1, h_2$ of $\mathcal{G}$ and a $(h_1, h_2)$-bicorn sequence $\mathcal{B}$. 
We first consider the case where $5 \leq t \leq 8$. 
Since the length of $\mathcal{G}$ is at most $8$, $\mathcal{B}$ is contained in the $5$-neighborhood of $\mathcal{G}$ by Lemma \ref{dist_bc}. 
Moreover, since every vertex of $\mathcal{G}$ is at least $10$ away from $\beta$, each $(h_1, h_2)$-bicorn curve in $\mathcal{B}$ is at least $5$ away from $\beta$. 
By the same argument as in the proof of Lemma \ref{bgit_far_away}, the assertion holds. 

In the case where $t = 4$, $\mathcal{B}$ is contained in the $3$-neighborhood of $\mathcal{G}$ by Lemma \ref{dist_bc}. 
Hence, the assertion holds. 
\end{proof}

\begin{proof}[Proof of Theorem \ref{bgit} (1)]
Let $\beta$ be a boundary component of $F'$. 
We may assume that $\mathcal{G}$ is not contained in the $17$-neighborhood of $\beta$. 
Pick any vertices $\gamma_1, \gamma_2$ in $\mathcal{G}$. 
We first consider the case where $d_{F} (\gamma_i, \beta) \geq 18$ ($i=1,2$).  
Let $g_1$ and $g_2$ be the vertices in $\mathcal{G}$ such that $d_{F} (g_i, \beta) = 18$, and that $g_i$ is nearest to $\gamma_i$. 
By Lemma \ref{bgit_far_away}, it follows that every component of $\pi_{F'} (\gamma_i)$ is at most 3 away from a component of $g_i$  (for $i=1,2$). 
We desire to estimate the bound for $d_{F'}(\pi(g_1), \pi(g_2))$. 
Since $d_F(g_1, g_2) = 36$, the length of the subpath $\mathcal{P}$ of $\mathcal{G}$ which connects $g_1$ and $g_2$ is equal to $36$. 
Since $d_{F} (g_i, \beta) = 18$ ($i=1,2$), the head and bottom subpaths of $\mathcal{P}$ of length $8$ satisfy the assumption of Lemma \ref{bgit_far_away_2}. 
By applying Lemma \ref{bgit_far_away_2} to these subpaths, we have that each pair of the components in $\pi_{F'}(g_1) \cup \pi_{F'}(g_2)$ is distanced at most $3+ (36 - 8 \times 2) + 3 = 26$. 
Consequently, we have $\mathrm{diam}_{F'}(\pi(\gamma_1) \cup \pi(\gamma_2)) \leq 3+26 +3 = 32$. 

We next treat the case where $d_{F} (\gamma_1, \beta) \geq 18$ and $d_{F} (\gamma_2, \beta) \leq 17$. 
By the same argument as above, we have $\mathrm{diam}_{F'}(\pi(\gamma_1) \cup \pi(\gamma_2)) \leq 3 + 3 + 10  + 10 + 3 \leq 29$. 

In the remainder case, the geodesic has length $\leq 17 + 17 = 34$, and so it follows from Lemma \ref{bgit_far_away_2} that $\mathrm{diam}_{F'}(\pi(\gamma_1) \cup \pi(\gamma_2)) \leq 3 + 10 + 10 + 3 = 26$. 
We thus conclude that the diameter of the non-annular subsurface projection is bounded above by $32$.  
\end{proof}

\section{An alternative proof of the bounded geodesic image theorem: annular projections} \label{annular}

When a subsurface $A$ of $F$ is an essential annulus, we need to employ a complicated definition of $\mathscr{AC}(A)$ (in order to make $\mathscr{AC}(A)$ quasi-isometric to $\mathbb{Z}$). 
Since the core curve of $A$ represents an non-trivial element of $\pi_1(F)$, we have the annular cover $Y$ of $F$ to which $A$ lifts homeomorphically. 
By using a hyperbolic metric of $F$ with finite area, we are able to have a compactification $\hat{Y}$ of $Y$. 
Then the {\it arc and curve graph} $\mathscr{AC}(A)$ is the graph whose vertices are the isotopy classes of the proper arcs of $\hat{Y}$ joining exactly two boundary components, and two vertices are connected by an edge if and only if they have representatives whose interiors are disjoint. 
Here, the isotopy in $\hat{Y}$ must fix the points of $\partial \hat{Y}$. 
As shown in \cite{MM00}, $\mathscr{AC}(A)$ is quasi-isometric to $\mathbb{Z}$, and that $d_{A}(\alpha, \beta) = 1 + |\alpha \cdot \beta|$, where $|\alpha \cdot \beta|$ is the algebraic intersection number of $\alpha$ and $\beta$. 
For the homotopy class of an essential simple closed curve $\alpha$,  the subsurface projection $\pi_{A} (\alpha)$ is defined to be the set of the fibers of $\alpha$ which connect the boundary components of $\hat{Y}$. 
Note that $\pi_A (\alpha)$ consists of pairwise adjacent vertices of $\mathscr{AC}(A)$. 
We say that $\alpha$ {\it cut} $A$ if $\pi_A (\alpha)$ is not empty. 

We use the following lemma due to Webb. 

\begin{lemma}{\cite[Lemma 4.1.3]{W14}} \label{Webb_lemma}
Let $Z$ be a non-annular subsurface of $F$, which contains $A$ as an essential subsurface. 
Suppose that $\alpha_0, \alpha_1, \ldots, \alpha_L$ are essential simple closed curves on $F$, whose subsurface projections on $Z$ have elements $a_i \subset \alpha_i $ such that $d_{Z}(a_i, a_{i+1}) = 1$. 
Then we have
$$d_{A} (\pi_{A} (\alpha_0), \pi_{A} (\alpha_{L})) \leq L + 4. $$ 
\end{lemma}
\begin{proof}
Webb's proof (including Masur--Schleimer's inequality \cite[Line 14]{MS13}) is valid even in the non-orientable cases. 
\end{proof}

\begin{lemma}\label{ann_lemma}
Let $Z$ be a non-annular subsurface of $F$, which contains $A$ as an essential subsurface. 
Let $\beta$ be a component of $\partial Z$. 
Suppose that $\mathcal{G}$ is a geodesic in $\mathscr{C}(F)$ whose all elements are at least 18 away from a component $\beta$. 
For each pair of elements $\gamma_1, \gamma_2 \in \mathcal{G}$, it follows that every component of $\pi_A (\gamma_1)$ is at most $8$ away from a component of $\pi_{A} (\gamma_2)$. 
\end{lemma}
\begin{proof}
Pick two vertices $\gamma_1, \gamma_2 \in \mathcal{G}$. 
By the proof of Lemma \ref{bgit_far_away}, there exist $(\gamma_1, \gamma_2)$-bicorn curves $b_1 , b_2$ such that $\pi_Z (\gamma_1) \cap \pi_Z (b_1)$ and $\pi_Z (b_2) \cap \pi_Z (\gamma_2)$ are not empty. 
Since $i (b_1, b_2) \leq 2$, we have an essential simple closed curve $\eta$ on $F$ such that $\eta$ is disjoint from $b_1 \cup b_2$ by Lemma \ref{start_log}. 
Since $\mathcal{G}$ is at least 18 away from $\beta$ and since $\beta$ is adjacent to the core curve of $A$ in $\mathscr{C}(F)$, we see that all of $b_1, \eta, b_2$ is at least 3 away from the core curve of $A$ by Proposition \ref{12ball_prop}. 
This means that each element of the sequence $\gamma_1, b_1, \eta, b_2, \gamma_2$ has non-empty annular projection. 
Moreover, the subsurface projection of the sequence onto $Z$ forms a path of length $\leq 4$ in $\mathscr{A}\mathscr{C} (Z)$. 
By Lemma \ref{Webb_lemma}, we obtain the assertion. 
\end{proof}

We prove Theorem \ref{bgit} (2). 

\begin{proposition}
The annular subsurface projection of every geodesic in $\mathscr{C}(F)$ whose all members cut the annulus $A$ has diameter $\leq 52$. 
\label{ann_proj_bd}
\end{proposition}
\begin{proof}
Pick a geodesic in $\mathscr{C}(F)$ whose all members cut $A$. 
Choose a subsurface $Z \subset F$ containing $A$ and its boundary component $\beta \subset \partial Z$. 
If the geodesic is contained in the 17-neighbourhood of $\beta$, then the diameter is bounded by $34$ and so the assertion holds. 
We may assume that the geodesic is not contained in the 17-neighbourhood of $\beta$. 
By Lemma \ref{ann_lemma} and an argument similar to  the proof of Theorem \ref{bgit}, we can deduce that the diameter of the annular subsurface projection is bounded by $ \leq 8 + 18 + 18 + 8 = 52$. 
\end{proof}

\begin{remark}
Our methodologies relies heavily on the implication $i(\alpha, \beta) \leq 2 \Rightarrow d_{F}(\alpha, \beta) \leq 2$, which does not hold when $F$ is a closed non-orientable surface of genus $3$ (i.e.\ $\chi(F) = -1$). 
\end{remark}

\begin{remark}
From Theorem \ref{bgit}, subsurface projection bounds of the original bounded geodesic image theorem due to Masur--Minsky for closed surfaces can be obtained as follows. 
Suppose that $F$ is orientable and that a subsurface $F'$ is not homeomorphic to an annulus. 
Consider the retraction $r \colon \mathscr{A} \mathscr{C} (F') \to \mathscr{C}(F')$ defined in \cite[Remark 5.1]{HPW15}. 
Since $r$ is 2-Lipschiz, any geodesic of length $\leq 32$ in $\mathscr{A} \mathscr{C} (F')$ is retracted onto a subset in $\mathscr{C}(F')$ with diameter $\leq 64$. 
Hence, $64$ is an upper bound for the original bounded geodesic image theorem. 
We note that, if $F'$ is an annulus, our subsurface projection is the same as that of Masur--Minsky. 
In conclusion, $64$ (resp.\ 52) is an upper bound for non-annular (resp.\ annular) subsurface projections in the original bounded geodesic image theorem. 

We also note that our bounds for annular and non-annular subsurface projections have higher accuracy than the bounds estimated by Webb  \cite[Theorem 4.2.1]{W14}. 
However, he also treated the subsurface projections from surfaces with boundary. 
\end{remark}

%%%%%%%% Appendix %%%%%%%%%
\section{Appendix: A hyperbolic constant for augmented curve graphs}\label{add}

By slightly modifying the proof for $\delta$-hyperbolicity of the curve graphs, we are able to estimate hyperbolic constant of augmented curve graphs. 
Fix $k \geq 2$. 
Recall that the augmented curve graph $G_k (F)$ of a surface $F$ is the graph such that vertices are the same as the curve graph, and that the edges are spanned by two vertices which have representatives intersecting at most $k$. 
By $d_G$ we denote the usual length metric in $G_k (F)$ . 
We note that the bicorn sequences are actual paths in $G_k (F)$. 

\begin{proposition} \label{aug_dist}
Let $\alpha, \beta$ be essential simple closed curves on $F$ with $i(\alpha, \beta) \geq 1$. 
Then we have 
$$d_G(\alpha, \beta) \leq \log_{k + 1} \left(i(\alpha, \beta) \right) + 1.  $$
\end{proposition}
\begin{proof}
Induction on $i(\alpha, \beta)$. 
By the definition of $G_k (F)$, any pair of essential simple closed curves $\alpha$ and $\beta$ intersecting at most $k$ is adjacent. 
Hence, the assertion is immediate when $i(\alpha, \beta) \leq k$. 

Suppose $i(\alpha, \beta) \geq k+1$. 
By Lemma \ref{1-3rd}, we have an $(\alpha, \beta)$-bicorn curve $\gamma$ such that $i (\alpha, \gamma) \leq 2$ and $i(\beta, \gamma) \leq \frac{i (\alpha, \beta)}{3}$. 
We have $d_G (\alpha, \beta) \leq d_{G} (\alpha, \gamma) + d_{G} (\beta, \gamma) \leq 1 + (\log_{k+1} (\frac{i(\alpha, \beta)}{3}) + 1)$, as desired. 
\end{proof}

It is interesting to note that the bound in Proposition \ref{aug_dist} is very similar to the well-known bound, $\log_{2} (i(\alpha, \beta)) + 1,$ for the Farey graph which is the curve graph of a torus. 

We prepare lemmas to estimate a hyperbolic constant of $G_{k} (F)$. 

\begin{lemma}\label{dist_bc_2}
Let $x_0, x_1, \ldots, x_m$ be a sequence of curves in $G_k (F)$ with $i(x_i, x_{i+1}) \leq 2$ $(1 \leq i \leq m-1)$ and $2^{n-1} < m \leq 2^{n}$ for some positive integer $n$. 
Suppose that $b$ is an $(x_0, x_m)$-bicorn curve. 
Then there exists a number $l$ such that $d_G (b, x_l) \leq n + 1 =  \lceil \mathrm{log}_2 (m) \rceil + 1$.  
\end{lemma}
\begin{proof}
We may assume that $m = 2^n$. 
Case $n = 1$. 
Since both of $x_0$ and $x_m$ intersect $x_1$ at most twice, it follows that $i (b, x_1) \leq 4$. 
By Lemma \ref{1-3rd}, we have a $(b, x_1)$-bicorn curve adjacent to both of $b$ and $x_1$. 
This implies $d_{G} (b, x_1) \leq 2$. 
By induction on $n$, it is easy to see that the assertion holds. 
\end{proof}

\begin{lemma} \label{7ball_prop}
Suppose that $\mathcal{G}$ is a geodesic connecting essential simple closed curves $\alpha$ and $\beta$ in $G_k(F)$, and $\mathcal{B}$ is an $(\alpha, \beta)$-bicorn sequence. 
Then $\mathcal{B}$ is contained in the $7$-neighborhood of $\mathcal{G}$. 
\end{lemma}
\begin{proof}
Let $b^{*}$ be a bicorn curve in $\mathcal{B}$ which is the furthest vertex from $\mathcal{G}$. 
Set $t:= d_{F}(b^{*}, \mathcal{G})$. 
By an argument similar as in the proof of Proposition \ref{12ball_prop}, the inequality $t \leq \lceil \mathrm{log}_2 (8t) \rceil + 1 = \lceil \mathrm{log}_2 (t) \rceil + 4$ holds. 
This implies $t \leq 7$. 

In fact, we can argue as follows. 
We may assume that neither $\alpha$ nor $\beta$ is contained in the closed $(2t)$-ball of $b^{*}$ (the other cases can be treated similarly).  
Pick a bicorn curve $\gamma_i$ in $\mathcal{B}$ preceding to $b^{*}$ and a bicorn curve $\gamma_j$ in $\mathcal{B}$ suceeding to $b^{*}$, which are contained in $S(b^{*}; 2t)$. 
Let $\mathcal{G}_i$ (resp.\ $\mathcal{G}_j$) be a geodesic of length $\leq t$ connecting $\gamma_i$ (resp.\ $\gamma_j$) and $\mathcal{G}$. 
By $g_i$ (resp.\ $g_j$) we denote the endpoint of $\mathcal{G}_i$ (resp.\ $\mathcal{G}_j$) which lies in $\mathcal{G}$. 
By $L$ we denote the concatenation of $\mathcal{G}_i$, $\mathcal{G}_j$ and the subpath $\mathcal{G}'$ of $\mathcal{G}$ which ends at $g_i$ and $g_j$. 
We note that the endpoints of $L$ are $\gamma_i$ and $\gamma_j$, and the length of $L$ is at most $t + (t+2t+2t+t) + t = 8t$. 
Applying Lemma \ref{dist_bc_2} to $b^{*}$ and $L$, it follows that $d_{F}(b^{*}, L) \leq  \lceil \mathrm{log}_2 (8t) \rceil + 1$. 
Pick a closed curve $l^{*}$ in $L$ which is the most nearest to $b^{*}$. 
In the case where $l^{*} \in \mathcal{G}'$, it holds that $k = d_{G}(b^{*}, l^{*}) \leq \lceil \mathrm{log}_2 (8t) \rceil + 1$, and the assertion holds. 
Therefore we consider the case where $l^{*} \in \mathcal{G}_i \cup \mathcal{G}_j$.  
Assume that $l^{*} \in \mathcal{G}_i$. 
Since $2t \leq d_{G}(b^{*}, \gamma_i) \leq d_{G}(b^{*}, l^{*}) + d_{G}(l^{*}, \gamma_i) \leq d_{G}(b^{*}, l^{*}) + t$, it holds that $d_{G}(b^{*}, l^{*}) \geq t$. 
On the other hand, since  $l^*$ is contained in $L$, it follows that $d_{G}(b^{*}, l^{*}) \leq  \lceil \mathrm{log}_2 (8t) \rceil + 1$. 
By the above argument, we have $ t \leq \lceil \mathrm{log}_2 (8t) \rceil + 1$, as desired. 
\end{proof}

Let us prove Theorem \ref{aug_hyp}. 

\begin{proof}[Proof of Theorem \ref{aug_hyp}]
Suppose that $\Delta$ is a geodesic triangle in $G_k(F)$. 
We may assume that the apices of $\Delta$ are sensible. 
By Lemmas \ref{bc_contain} and \ref{7ball_prop} and by Proposition \ref{slimbicorn_2}, we have that $\Delta$ has an $8$-center which is a bicorn curve between two apices. 
\end{proof}

\begin{proposition}
Suppose that $\mathcal{G}$ is a geodesic in $G_k(F)$ connecting essential simple closed curves $\alpha$ and $\beta$, and $\mathcal{B}$ is an $(\alpha, \beta)$-bicorn sequence. 
Then $\mathcal{G}$ is contained in the $14$-neighborhood of $\mathcal{B}$. 
\label{7ball_prop}
\end{proposition}
\begin{proof}
Pick a vertex $g$ of $\mathcal{G}$ which is not contained in the $7$-neighbouhood of $\mathcal{B}$. 
Consider the minimal subpath $\mathcal{P}_{g}$ of $\mathcal{G}$ containing $g$, whose endpoints are contained in the $7$-neighbouhood of $\mathcal{B}$. 
The length of $\mathcal{P}_g$ is bounded above by $7 + 1 + 7 = 15$. 
By the definition of $\mathcal{P}_{g}$, we have that $g$ is at most $14$ away from an $(\alpha, \beta)$-bicorn curve nearest to an endpoint of $\mathcal{P}_{g}$. 
\end{proof}

\end{document}